\documentclass[12pt,a4paper]{article}
\usepackage[OT1]{fontenc}
\usepackage[latin1]{inputenc}
\usepackage[english]{babel}
\usepackage{amsfonts}
\usepackage{amsthm}
\newtheorem{theo}{Theorem}[section]
\newtheorem{prop}[theo]{Proposition}
\newtheorem{lemm}[theo]{Lemma}
\newtheorem{coro}[theo]{Corollary}
\newtheorem{defi}[theo]{Definition}
\newtheorem{rema}[theo]{Remark}

\begin{document}

\title{\vspace*{0cm}Nonpositively curved metric in the positive cone of a finite von Neumann algebra}

\date{}
\author{Esteban Andruchow and Gabriel Larotonda\footnote{Partially supported by IAM-CONICET.}}

\maketitle

%\begin{spacing}{.95}
\abstract{In this paper we study the metric geometry of the space $\Sigma$ of positive invertible elements of a von Neumann algebra ${\cal A}$ with a finite, normal and faithful tracial state $\tau$. The trace induces an incomplete Riemannian  metric $<x,y>_a=\tau (ya^{-1}xa^{-1})$, and though the techniques involved are quite different, the situation here resembles in many relevant aspects that of the $n\times n$ matrices when they are regarded as a symmetric space. For instance we prove that geodesics are the shortest paths for the metric induced, and that the geodesic distance is a convex function; we give an intrinsic (algebraic) characterization of the geodesically convex submanifolds $M$ of $\Sigma$, and under suitable hypothesis we prove a factorization theorem for elements in the algebra that resembles the Iwasawa decomposition for matrices. This factorization is obtained \textit{via} a nonlinear orthogonal projection $\Pi_M:\Sigma\to M$, a map which turns out to be contractive for the geodesic distance.
\footnotesize{\noindent }}
%\end{spacing}

\setlength{\parindent}{0cm} %% para que no indente los parrafos nuevos

\section{Introduction}

Let ${\cal A}$ be a von Neumann algebra with a finite (normal, faithful) trace $\tau$. Denote by ${\cal A}_h$ the set of selfadjoint elements of ${\cal A}$,  by $G_{\cal A}$ the group of invertible elements, and by $\Sigma$ the set
\[
\Sigma=e^{{\cal A}_h}=\{ a\in G_{\cal A}: a\ge 0\};
\]
$\Sigma$ is an open subset of ${\cal A}_h$ in the norm topology. Therefore if one regards it as a manifold, its tangent spaces identify with ${\cal A}_h$. We endow these tangent spaces with the (incomplete) Hilbert-Riemann metric
\begin{equation}\label{metrica}
<x,y>_a=\tau(xa^{-1}ya^{-1}) , \ \ a\in \Sigma , \ x,y\in {\cal A}_h.
\end{equation}

Note that $\|x\|_a^2=<x,x>_a:=\tau(xa^{-1} xa^{-1})$, and also that this metric is invariant for the action $I_g: x\mapsto gxg^*$, where $g\in G_{\cal A}$.

\medskip

As in classical differential geometry, one obtains a metric $d$ for $\Sigma$ by considering
\begin{equation}\label{dis}
dist(a,b)= \inf\{Length(\gamma): \gamma \hbox{ is a smoooth curve joining } a  \hbox{ and } b\}, 
\end{equation}
where smooth means differentiable in the norm induced topology and the length of a curve $\gamma(t)$, $t\in [0,1]$ is measured using the inner product above (\ref{metrica}):
\[
Length(\gamma)=\int_0^1 <\dot{\gamma}(t),\dot{\gamma}(t)>_{\gamma(t)}^{\frac12}\; dt.
\]
The purpose of this paper is the geometric study of the resulting metric space, and particularly, of its convex subsets.

\medskip

If ${\cal A}$ is finite dimensional, i.e. a sum of matrix spaces, this metric is well known: it is the non positively curved Riemannian metric on the set of positive definite matrices, which is a universal model space for (finite dimensional) non positively curved manifolds on non compact type (see \cite{eberlein} and \cite{mostow}).

\medskip

If ${\cal A}$ is of type II$_1$, the trace inner product is not complete, so that $\Sigma$, with the inner products  $<\ ,\ >_a$, is not a Hilbert-Riemann manifold properly speaking. For instance, the exponential map
\[
exp:{\cal A}_h\to \Sigma , \ \ exp(x)=e^x,
\]
which is a global diffeomorphism in the norm topology, is continuous but non differentiable in the 2-norm $\| \ \|_2$ induced by $\tau$ (namely $\|x\|_2=\tau(x^*x)^{\frac12}$). The set $\Sigma$ itself is not a differentiable manifold with this norm. 

\medskip

However, the metric space $(\Sigma,dist)$  behaves in many senses like in the finite dimensional setting. Let us mention a few issues:
\begin{enumerate}
\item
The isometric action of the group $G_{\cal A}$ via $g\mapsto I_g$, where $I_g(x)=gxg^*$
\item
Minimality of geodesics (i.e. solutions of Euler's equation are minimizing for the distance introduced above in (\ref{dis}), see Theorem \ref{mini})
\item
Convexity of the map $t\mapsto dist(\gamma(t),\delta(t))$ which gives distance among geodesics (Corollary \ref{conv})
\item
Algebraic structure of (geodesically) convex subsets (Theorem \ref{alge}).
\item
Normal projections to convex submanifolds and their minimality (Lemma \ref{normi} and Theorem \ref{proje})
\item 
Existence and uniqueness of a factorization for invertible elements by means of convex submanifolds (Corollary \ref{invert})
\end{enumerate}

\section{Main inequalities}

The following inequality will be useful; its proof for $n\times n$ real matrices can be found in the inspiring paper of G.D. Mostow \cite{mostow}. It is called by R. Bhatia \cite{bhatia} the {\it exponential metric increasing property}. Bhatia proves it for matrices (and for more general norms). However his proof for the $2$-norm is valid almost verbatim in the infinite dimensional context for an arbitrary (finite, faithful) tracial state. We transcribe it. We use the fact that selfadjoint elements in a von Neumann algebra  can be approximated by selfadjoint elements with finite spectrum. 

\begin{lemm}
Let $\tau$ be a tracial faithful state in ${\cal A}$, and $x,y$ selfadjoint elements of ${\cal A}$. If $exp$ denotes the usual exponential map, $exp(x)=e^x$, then
\begin{equation}\label{emip}
\|y\|_2\le \|e^{-x}dexp_x(y)\|_2.
\end{equation}
\end{lemm}
\begin{proof}
First we must establish the formula
\[
dexp_x(y)=\int_0^1 e^{tx}y e^{(1-t)x} dt.
\]
Note that $dexp_x(y)=\frac{d}{dt} e^{x+ty}|_{t=0}$. Then
\[
dexp_x(y)= y+\frac12(yx+xy)+\frac16(yx^2+xyx+x^2y)+\dots
\]
On the other hand,
\[
e^{tx}y e^{(1-t)x}=y+txy+(1-t)yx+\frac12(1-t)^2yx^2+t(1-t)xyx+\frac12 t^2x^2y+\dots
\]
Integrating this  series (which is absolutely convergent) term by term proves the equality. Denote $a=e^x$.
Let us show now that if $b$ is positive in ${\cal A}$,
\begin{equation}\label{desigualdad}
\|a^{\frac12}ba^{\frac12}\|_2\le \|\int_0^1 a^t b a^{1-t}dt\|_2.
\end{equation}
Assume first that $a$ has finite spectrum: $a=\sum_{i=1}^n \alpha_i p_i$, with $\alpha_i>0$ and $\sum_{i=1}^np_i=1$. Then $a^{\frac12}ba^{\frac12}=\sum_{i,j=1}^n \alpha_i^{\frac12}\alpha_j^{\frac12}p_ibp_j$. Therefore
\[
\|a^{\frac12}ba^{\frac12}\|_2^2=\sum_{i,j=1}^n \alpha_i\alpha_j\tau(p_i b p_j b p_i).
\]
Analogously $\int_0^1 a^tba^{1-t}dt=\sum_{i,j}^n\int_0^1 \alpha_i^t\alpha_j^{1-t}dt\ p_i b p_j$ and
\[
\|\int_0^1 a^tba^{1-t}dt\|_2^2=\sum_{i,j}^n\int_0^1 \alpha_i^{2t}\alpha_j^{2(1-t)}dt \  \tau(p_i b p_j b p_i)=\sum_{i,j}^n {\frac{\alpha_i^2-\alpha_j^2}{2\ln \alpha_i-2\ln \alpha_j}}\tau(p_ibp_jbp_i).\]
Note that $p_ibp_jbp_i$ is positive. Also one has the elementary inequality
\[
\sqrt{st}\le \frac{s-t}{\ln s-\ln t}
\]
for $s,t>0$. Then 
\[
\alpha_i\alpha_jp_ibp_jbp_i\le \frac{\alpha_i^2-\alpha_j^2}{2\ln \alpha_i-2\ln \alpha_j}p_ibp_jbp_i.
\]
Taking traces and adding yields (\ref{desigualdad}) in this case. In the general case, the inequality follows by approximating (in norm) the element $a$ with a positive elements with finite spectrum.

As in \cite{bhatia}, put $b=e^{-x/2}ye^{-x/2}$ in (\ref{desigualdad}):
\[
\|y\|_2\le \|\int_0^1 e^{tx}(e^{-x/2}ye^{-x/2})e^{(1-t)x} dt\|_2=\|e^{-x/2}\int_0^1 e^{tx} y e^{(1-t)x} dt \ e^{-x/2}\|_2
\]
\[
\hspace*{-5.4cm}=\|e^{-x/2}(dexp_x(y))e^{-x/2}\|_2.
\]
If $a$ is positive and invertible and $b$ is selfadjoint, by the Cauchy-Schwarz inequality for  $\tau$, one has
$$
\|a^{-\frac12}ba^{-\frac12}\|_2^2=\tau(a^{-1}ba^{-1}b)\le \tau(a^{-1}b^2a^{-1})^{\frac12}\tau(ba^{-2}b)^{\frac12}=\|a^{-1}b\|_2^2.
$$
Using this inequality for $a=e^x$ and $b=dexp_x(y)$ one obtains
$$ 
\|y\|_2\le \|e^{-x/2}(dexp_x(y))e^{-x/2}\|_2\le \|e^{-x}(dexp_x(y))\|_2. 
$$
\end{proof}

\smallskip

\begin{coro}\label{tx}
For any $x\in {\cal A}_h$, the map $\;T_x: y \mapsto e^{-x/2}dexp_x(y)e^{-x/2}$ is bounded, symmetric for the $2$-inner product (when restricted to ${\cal A}_h$) and invertible. The inverse is contractive i.e $\;T_x^{-1}(z)\|_2\le \|z\|_2.$\end{coro}
\begin{proof}
The map is clearly bounded and invertible, the bound for the inverse follows from the proof of the previous Lemma. To prove that it is symmetric, note that
\[
<T_x(y),z>_2=\tau(zT_x(y))=\tau(  e^{-x/2}  \sum_{n\ge 0}\frac{1}{n!}\sum\limits_{p+q=n-1} x^pyx^q  e^{-x/2} z)=\qquad\qquad\quad
\]
\[
= \sum_{n\ge 0}\frac{1}{n!}\sum\limits_{p+q=n-1} \tau(  e^{-x/2}  x^p y x^q  e^{-x/2} z)=\sum_{n\ge 0}\frac{1}{n!}\sum\limits_{p+q=n-1} \tau(   x^p e^{-x/2} y e^{-x/2}   x^q   z)=
\]
\[ =\sum_{n\ge 0}\frac{1}{n!}\sum\limits_{p+q=n-1} \tau(  e^{-x/2}  x^q z x^p  e^{-x/2}   y)=  \tau(T_x(z) y)=<y,T_x(z)>_2.
\]
\end{proof}

\section{Geodesic distance}

For $X,Y$ smooth vector fields in $\Sigma$ and $p\in \Sigma$, we introduce the expression
\begin{equation}\label{covariant}
\left(\nabla_X Y\right)_p=\{X(Y)\}_p-\frac{1}{2}\left( X_p\; p^{-1}\; Y_p+Y_p\;p^{-1}\; X_p  \right)
\end{equation}
where $X(Y)$ denotes derivation of the vector field $Y$ in the direction of $X$ (performed in the linear space ${\cal A}_h$). Note that $\nabla$ is clearly symmetric and verifies all the formal identities of a connection. The compatibility condition between the connection and the metric
\[
\frac{d}{dt}<X,Y>_{\gamma}=<\nabla_{\dot\gamma} X,Y>_{\gamma}+<X,\nabla_{\dot\gamma} Y>_{\gamma}
\]
is fulfilled for any smooth curve $\gamma \subset\Sigma$ and $X,Y$ tangent vector fields along $\gamma$. This identity is straightforward from the definitions for both terms and the cyclicity of the trace. This says that $\nabla$ is the "Levi-Civita" connection of the metric introduced.

\smallskip

Euler's equation $\nabla_{\dot \gamma}\dot\gamma=0$ reads $\ddot\gamma\ = \dot\gamma\gamma^{^{-1}}\dot\gamma$, 
and it is easy to see that the (unique) solution of this equation with $\gamma(0)=p$, $\gamma(1)=q$ is given by the curve
\begin{equation}\label{curvas}
\delta_{pq}(t)=p^{\frac{1}{2}}\left( p^{-\frac{1}{2}}q p^{-\frac{1}{2}}\right)^t p^{\frac{1}{2}}.
\end{equation}
Note that $\delta_{pq}\subset\Sigma$ because $aba$ is positive invertible whenever $a,b$ are positive invertible.

\smallskip

We will prove that the shortest path joining $p$ to $q$ is given by the formula above  (Theorem \ref{mini}); these curves look formally equal to the geodesics between positive definite matrices (regarded as a symmetric space). 

\smallskip

We will use ${\rm Exp}_p$ to denote the exponential map of $\Sigma$. Note that 
\[
{\rm Exp}_p(v)=p^{\frac12}\;{\rm e}^{\;p^{-\frac12}\,v\,p^{-\frac12}}p^{\frac12}.
\]
Rearranging the exponential series we get a simpler expression
\[
{\rm Exp}_p(v)=p\;{\rm e}^{p^{-1}v}={\rm e}^{\,vp^{-1}}p.
\]
A straightforward computation also shows that for $p,q\in\Sigma$ we have 
\[
{\rm Exp}_p^{-1}(q)=p^{\frac12}\ln(p^{-\frac12}\,q\,p^{-\frac12})p^{\frac12}.
\]

As mentioned in the introduction, we measure curves in $\Sigma$ using the norms in the tangent space, namely
\[
Length(\alpha)=\int_0^1\|\dot\alpha(t)\|_{{\alpha}(t)}\, dt.
\]

We have $\|\dot{\delta_{p,q}}  (t)\|_{\delta_{p,q}(t)}\equiv \|\ln(p^{-\frac12}qp^{-\frac12})\|_2
$, so for the geodesics introduced in equation (\ref{curvas}), we have $L(\delta_{p,q})=\|\ln(p^{-\frac12}qp^{-\frac12})\|_2$.

\smallskip

\begin{theo}\label{mini}
Let $a,b\in \Sigma$. Then the geodesic $\delta_{a,b}$ is the shortest curve joining $a$ and $b$ in $\Sigma$, if the length of curves is measured with the metric defined above.
\end{theo}
\begin{proof}
Let $\gamma$ be a smooth curve in $\Sigma$ with $\gamma(0)=a$ and $\gamma(1)=b$. We must compare the length of $\gamma$ with the length of $\delta_{a,b}$. Since the invertible group acts isometrically for the metric, it preserves the lengths of curves. Thus we way act with $a^{-\frac12}$, and suppose that both curves start at $1$, or equivalently, $a=1$. Therefore $\delta_{1,b}(t)=\delta(t)=e^{t x}$, with $x=\ln b$. The length of $\delta$ is therefore $\tau(x^2)^{\frac12}=\|x\|_2$. 
The proof follows easily from the inequality proved above. Indeed, since $\gamma$ is a smooth curve in $\Sigma$, it is of the form $\gamma(t)=e^{\alpha(t)}$, with $\alpha=\ln \gamma$. Then $\alpha$ is a smooth curve of selfadjoints with $\alpha(0)=0$ and $\alpha(1)=x$.
Moreover,
\[
\tau((\gamma^{-1}\dot{\gamma}\gamma^{-1}\dot{\gamma})^{\frac12}=\|e^{-\alpha} \dot{e^{\alpha}}\|_2=
\|e^{-\alpha} dexp_\alpha(\dot{\alpha})\|_2.
\]
By the inequality in the above lemma, this is not smaller than $\|\dot{\alpha}\|_2$.
Then
\[ \int_0^1 \tau((\gamma^{-1}\dot{\gamma}\gamma^{-1}\dot{\gamma})^{\frac12} dt \ge \int_0^1 \|\dot{\alpha}\|_2 dt \ge {\big\|}\int_0^1 \dot{\alpha} \  dt\ {\big\|}_2=\|x\|_2=\tau(x^2)^{\frac12}.
\]
\end{proof}

\begin{rema}
The geodesic distance induced by the metric is given by
\[
dist(a,b)=\tau\Bigl(\ln\bigl(a^{-\frac12}ba^{-\frac12}\bigr)^2\Bigr)^{\frac12}.
\]
\end{rema}

The curvature tensor \cite{cprpositive} is given by
\[
R_a(x,y)z=-\frac14 a\bigl[[a^{-1}x,a^{-1}y],a^{-1}\bigr]
\]
where $[\ ,\ ]$ is the usual commutator, {\it i.e} $[x,y]=xy-yx$. 

\medskip

Let $J(t)$ be a Jacobi field along a geodesic $\delta$ of $\Sigma$. That is, $J$ is a solution of the differential equation
\begin{equation}\label{ecuacionjacobi}
\frac{D^2J}{dt^2}+R_\delta(J,\dot{\delta})\dot{\delta}=0.
\end{equation}

Next we show that the norm of a Jacobi field is convex. If $x,y\in {\cal A}_h$ are regarded as tangent vectors of $\Sigma$ at the point $a$, then the following condition (which is a non positive sectional curvature condition) holds:
\[
<R_a(x,y)y,x>_a=\tau(R_a(x,y)ya^{-1}xa^{-1})\le 0.
\]
The proof of this fact is straightforward.
Then 
\[
\frac{d^2}{dt^2}<J,J>_{\gamma}=2\left\{<\frac{D^2 J}{dt^2},J>_{\gamma}+<\frac{D J}{dt},\frac{D J}{dt}>_{\gamma} \right\}=
\]
\[
\hspace*{3.7cm}=2\left\{-<R_\gamma(J,\dot{\gamma})\dot{\gamma},J>_{\gamma}+<\frac{D J}{dt},\frac{D J}{dt}>_{\gamma}\right\}\ge 0.
\]
In other words, the smooth function $t\mapsto <J,J>_{\gamma}$ is convex. We shall need convexity of the {\it norm} of the Jacobi field (and not of the {\it square} of the norm just noted).

\begin{prop}
Let $\gamma$ be a geodesic of $\Sigma$ and $J$ a Jacobi field along $\gamma$. The real map $t\mapsto <J,J>_{\gamma}^{\frac12}$ is convex.
\end{prop}
\begin{proof}
Clearly, is suffices to prove this assertion for a field $J$ which does not vanish. As in Theorem 1 of \cite{cprjacobi}, by the invariance of the connection and the metric under the action of $G_{\cal A}$, it suffices to consider the case of a geodesic $\gamma(t)=e^{tx}$  starting at $1\in \Sigma$ ($x\in {\cal A}_h$). For the field $K(t)=e^{-tx/2}J(t)e^{-tx/2}$ the Jacobi equation translates into
\begin{equation}\label{jacobimodificada}
4\ddot{K}=Kx^2+x^2K-2xKx.
\end{equation}
Moreover
\[
<J,J>_{\gamma}^{\frac12}=\tau(\gamma^{-1}J\gamma^{-1}J)^{\frac12}=\tau(K^2)^{\frac12}=\|K\|_2.
\]
Let us prove therefore that the map $t\mapsto f(t)=\|K(t)\|_2$ is convex, for any (non vanishing) solution $K$ of (\ref{jacobimodificada}). Note that $f(t)$ is smooth, and $\dot{f}=\tau(K^2)^{-\frac12} \tau(K\dot{K})$.
Then
\[
\ddot{f}=-\tau(K^2)^{-\frac32}\tau(K\dot{K})^2+\tau(K^2)^{-\frac12}\{\tau(\dot{K}^2)+\tau(K\ddot{K})\}.
\]
Let us multiply this expresion by $\tau(K^2)^{\frac32}$ to obtain
\[
-\tau(K\dot{K})^2+\tau(K^2)\tau(\dot{K}^2)+\tau(K^2)\tau(K\ddot{K}).
\]
The first two terms add up to a non negative number. Indeed, one has $\tau(K\dot{K})^2\le \tau(K^2)\tau(\dot{K}^2)$ by the Cauchy-Schwarz inequality for the trace $\tau$. Let us examine the third term $\tau(K^2)\tau(K\ddot{K})$. It suffices to show that $\tau(K\ddot{K})$ is non negative. Using (\ref{jacobimodificada}),
\[
\tau(K\ddot{K})=\frac14\{\tau(K^2x^2)+\tau(Kx^2K)-2\tau(KxKx)\}=\frac12\{\tau(K^2x^2)-\tau(KxKx)\}.
\]
This number is positive, again by the Cauchy-Schwarz inequality:
\[
\tau(KxKx)=\tau((xK)^*Kx)\le \tau((xK)^*xK)^{\frac12}\tau((Kx)^*Kx)^{\frac12}=\tau(K^2x^2). 
\]
\end{proof}

\smallskip

\begin{coro}\label{conv}
If $\gamma$ and $\delta$ are geodesics, the map $f(t)=dist(\gamma(t),\delta(t))$ is a convex function of $t$.
\end{coro}
\begin{proof}
As in Theorem 2 of \cite{cprjacobi}, distance between $\gamma(t)$ and $\delta(t)$ is given by the geodesic $\alpha_t(s)$ obtained moving the $s$ variable in a geodesic square $h(s,t)$ with vertices $\gamma(t_0),\delta(t_0),\gamma(t_1),\delta(t_1)$. Taking the partial derivative along the $s$ direction gives a Jacobi field $J(s,t)$ along the geodesic $\beta_s(t)=h(s,t)$ and it also gives the speed of $\alpha_t$. Hence 
\[
f(t)=\int_0^1\|\frac{\partial\alpha_t}{\partial s}(s)\|_{\alpha_t(s)}ds=\int_0^1 \| J(s,t)\|_{h(s,t)}\;ds.
\]
This equation says that $f(t)$ can be written as the limit of a convex combination of convex functions $u_i(t)= \| J(s_i,t)\|_{h(s_i,t)}$, so $f$ must be convex itself.
\end{proof}

\smallskip

\begin{lemm} The following inequality holds for any $x,y\in{\cal A}_h$:
\begin{equation}\label{des}
dist(e^x,e^y)=\|\ln\left(e^{x/2}e^{-y}e^{x/2}\right)\|_{2}\ge \|x-y\|_{2}.
\end{equation}
\end{lemm}
\begin{proof}
Take $\gamma(t)=e^{tx}$, $\delta(t)=e^{ty}$ and $f$ as in the previous corollary. Note that $f(0)=0$, hence $f(t)/t\le f(1)$ for any $0<t\le 1$; hence $\lim\limits_{t \to 0^+}f(t)/t\le f(1)$. Now note that 
\[
f(t)/t=\frac{1}{t}\|\ln\left(e^{tx/2}e^{-ty}e^{tx/2}\right)\|_{2}=\tau\left( \left[\frac{1}{t}\ln\left(e^{tx/2}e^{-ty}e^{tx/2}\right) \right]^2\right)^{\frac12}.
\]
Since $\lim_{t\to 0^+}\frac{1}{t}\ln\left(e^{tx/2}e^{-ty}e^{tx/2}\right)=\frac{d}{dt}\mid_{t=0}\ln\left(e^{tx/2}e^{-ty}e^{tx/2}\right)$, and the logarithm of $\beta(t)=e^{tx/2}e^{-ty}e^{tx/2}$ can be approximated uniformly by polinomials $p_n(\beta)=\sum_k\alpha_{n,k} \beta^k$ for $t$ close enough to zero (note that $\beta(0)=1$), and $\frac{d}{dt}\beta\mid_{t=0}=x-y$, we have the desired inequality.
\end{proof}

\smallskip

Using the inner product in each tangent space, we can talk about angles between curves and more general subsets of $\Sigma$ in a natural way; in particular we have:

\begin{lemm}\label{tri}
The sum of the inner angles of any geodesic triangle in the manifold $\Sigma$ is less or equal than $\pi$
\end{lemm}
\begin{proof}Squaring both sides of inequality (\ref{des}) leads (by the invariance of the metric for the action of $G_{\cal A}$) to
\[
l_i^2\ge l_{i+1}^2+l_{i-1}^2-2l_{i+1}l_{i-1}\cos(\alpha_i)
\]
where $l_i$ are the sides of any geodesic triangle and $\alpha_i$ is the angle opposite to $l_i$. 

These inequalities say that we can construct an Euclidian comparison triangle in the affine plane with sides $l_i$; they also say that the angle $\beta_i$ (opposite to $l_i$ for this flat triangle) is bigger than $\alpha_i$. Adding the three angles we have $\alpha_1+\alpha_2+\alpha_3\le \beta_1+\beta_2+\beta_3=\pi$. \end{proof}

\section{Convex sets}

We are interested in the convex subsets of $\Sigma$, that is, subsets $M\subset\Sigma$ such that the ambient geodesic joining two points in $M$ stays in $M$ for any value of $t$. Note that the simplest of such objects are the geodesics. 

It's not hard to see that when two elements $a,b\in \Sigma$ commute, the geodesic triangle spanned by $a,b$ and $1$ is convex, hence there is a flat surface containing $a,b$ and $1$; indeed, the triangle in $\Sigma$ is the image of the plane triangle with vertices $0,e_1,e_2$ under the map $T:\mathbb R^2\to \Sigma$ given by
\[
T(x,y)=e^{x\ln(a)-y\ln(b)}
\]
In particular the geodesic joining $a$ and $b$ is the image of the segment \[\{(x,y): x,y\ge 0,\;x+y=1\}.\]
This is not true in the general case (though the length of the segment is a lower bound for the length of that geodesic, as Lemma \ref{des} shows).

\smallskip

\begin{defi} An \emph{exponential set} $M\subset \Sigma$ is the exponential of a (closed, selfadjoint) subspace through the origin. In other words, $M=e^H$ with $H$ a closed subspace of ${\cal A}_h$. 
\end{defi}

\begin{lemm}
If $M$ is a convex exponential set in $\Sigma$, the geodesic symmetry $\sigma_ p:q\mapsto pq^{-1}p$ maps $M$ into $M$ for any $p\in M$
\end{lemm}
\begin{proof}
The map $\sigma_p$ maps any geodesic through $p$, $\gamma(t)=p^{\frac12}e^{tp^{-\frac12}vp^{-\frac12}} p^{\frac12}$ onto $\gamma(-t)$; now it is clear that it is an isommetry of $\Sigma$ and it maps $M$ into $M$.
\end{proof}

Note that, if $M$ is convex and $a\in M$, then $a^{\alpha}=e^{\alpha\ln a}\mbox{ is in }M\mbox{ for any real }\alpha$. This observation together with the previous lemma leads to the following characterization of convexity:

\begin{prop}
If $M$ is a convex, exponential set in $\Sigma$, then 
\begin{equation}\label{aba}
aba\in M\mbox{ whenever }a,b\in M
\end{equation}
\end{prop}
\begin{proof}
Note that $\displaystyle 
aba= a^{\frac32}\left(  a^{\frac12}b^{-1}a^{\frac12} \right)^{-1}a^{\frac32}=\sigma_{a^{\frac32}}\circ \sigma_{a^{\frac12}}(b)$.
\end{proof}

The converse of this last statement is also true (this can be easily seen iterating the property above in order to construct $\gamma(t)$ for given $p,q\in M$). 

\smallskip

Let us see how this property looks in the tangent $H$ (recall that $M=e^H$). This result is related to the results of \cite{cocoexp} by H. Porta and L. Recht:

\smallskip

\begin{theo}\label{alge}
If $H$ is a closed subspace of ${\cal A}_h$ (in the norm topology of $\cal A$), then  $M=e^H$ is a geodesically convex subset of $\Sigma$ if and only if $[x,[x,y]]\in H$ for any $x,y\in H$.
\end{theo}
\begin{proof}We use  property (\ref{aba}) above to identify convex sets; the proof follows the guidelines of \cite{mostow} for matrices, and we translate it here. 

We first assume $H$ has the double bracket property. Set $D_x:{\cal A}\to{\cal A}$, $D_x=L_x-R_x$, the difference between left and right multiplication by $x$ in ${\cal A}$. 

Let's consider the completion of $\cal A$ with respect to the trace, namely ${\cal H}=L^2({\cal A},\tau)$. Clearly ${\cal H}_{\mathbb R}=L_{\mathbb R}^2({\cal A}_{h},\tau)$ contains as a proper, closed subspace the completion of $H$, namely ${\cal H}_1=L_{\mathbb R}^2(H,\tau)$. Since $\tau$ is a normal trace, the involution $^*$ extends to a bounded antilinear operator $J$ of ${\cal H}_{\mathbb R}$, and the map $D_x$ extends uniquely to a bounded linear operator of ${\cal H}_{\mathbb R}$ (which we will still call $D_x$).

First we establish the identity
\begin{equation}\label{sinh}
T_x(y)=g(D_x/2)(y)
\end{equation}
where $T_x$ is the extension of the map from Corollary \ref{tx}, and 
\[g (z):=\frac{\sinh(z)}{z}=\sum_{n\ge 0} \frac{z^{2n}}{(2n+1)!}\]
is an entire function. Note that $g(z)=(2z)^{-1}(e^z-e^{-z})$. To prove (\ref{sinh}), we take derivative with respect to $t$ in the identity $X(t)e^{X(t)}=e^{X(t)}X(t)$, where $X(t)=x+ty$; after rearranging the terms we come up with
\[
\frac12\left(e^{D_x/2}-e^{-D_x/2}\right)y=(D_x\circ T_x)(y).
\]
Note that if $D_x$ were invertible, we would be set; this is not necessarily the case. However, $D_x^2=D_x\circ D_x$ is selfadjoint when restricted to ${\cal A}$, and since $T_x$ (more precisely, its extension) is also sefaldjoint (Corollary \ref{tx}), the operator $T=T_x-g(D_x/2)$ is selfadjoint on ${\cal A}$, and hence on ${\cal H}$ (note that $g(z)$ involves only even powers of $z$). The equation above says that we have proved that $(D_x\circ T)(y)=0$ for any $y\in{\cal A}$; in other words $T$ maps ${\cal H}$ into $\{x\}'=\overline{\{b\in{\cal A}_h: bx=xb\}}$. A straightforward computation shows that $Tb=0$ for any $b\in \{x\}'$, which proves that $T=0$, i.e. equation (\ref{sinh}) holds.

Now, for $x,y\in H$ consider the curve $e^{\alpha(t)}=e^{tx}e^ye^{tx}$. Clearly $\alpha(0)=y\in H\subset {\cal H}_1$; we will prove that $\alpha$ obeys a differential equation in ${\cal H}_{\mathbb R}$ which has a flow that maps ${\cal H}_1$ into ${\cal H}_1$, and by the uniqueness of the solution of such equation we will have $e^xe^ye^x=e^{\alpha(1)}\in e^H=M$.

Differentiating at $t=t_0$ the equation yields to 
\[
xe^{\alpha(t_0)}+e^{\alpha(t_0)}x=dexp_{\alpha(t_0)}(\dot\alpha(t_0))=e^{\alpha(t_0)/2}T_{\alpha(t_0)}(\dot\alpha(t_0))e^{\alpha(t_0)/2}=
\]
\[
\quad=e^{\alpha(t_0)/2}  g(D_{\alpha(t_0)}/2)(\dot\alpha(t_0)) e^{\alpha(t_0)/2}.
\]
Note that $g(z)$ is invertible whenever $z$ is a bounded linear operator, and also that the power series for $zcoth(z/2)$ involves only even powers of $z$. On the other hand, $D_z/2=D_{z/2}$ and $e^zxe^{-z}=e^{D_z}x$, hence
\[
\dot\alpha=g^{-1}(D_{\alpha}/2)\circ (e^{-\alpha/2}xe^{\alpha/2}+e^{\alpha/2}xe^{-\alpha/2})=\qquad\qquad\qquad\qquad\qquad\qquad\quad
\]
\[
=g^{-1}(D_{\alpha/2})\circ (e^{D_{{\alpha}}/2}+e^{-D_{{\alpha}}/2}) (x)=D_{\alpha}\coth(D_{\alpha/2})(x)=\sum_{n} c_n D_{\alpha}^{2n} x=
\]
\[
=\sum_{n}c_n D_{\alpha}^2\circ \cdots \circ D_{\alpha}^2(x)=F(\alpha).\hspace*{6.1cm}
\]
Since $D_z^2(x)=[z,[z,x]]$, $F(z)=\sum_n c_n D_{z}^{2n}(x)$ can be regarded as a map from ${\cal H}_1$ to ${\cal H}_1$, and since it is clearly an analytic map of ${\cal H}$ into ${\cal H}$, it fulfills a Lipschitz condition. Now the unique solution must be $\alpha(t)=\ln(e^{tx}e^ye^{tx})$. This proves that $M$ is convex whenever $H$ has the double bracket property.

To prove the other implication, assume $M=e^H$ is convex and $H$ is closed in the norm topology of ${\cal A}_h$. Clearly the path $\alpha$ stays in $H$ for any value of $t$ (here $e^{\alpha(t)}=e^{tx}e^ye^{tx}$), and the same is true for $\dot\alpha$. Now since $\dot\alpha(t)=D_{\alpha(t)}\coth(D_{\alpha(t)}/2)x$, we have
\[
\lim_{t\to 0} \frac{\dot\alpha(t)-\dot\alpha(0)}{t^2}=\lim_{t\to 0} \frac{(1+\frac{1}{12}t^2D_{\alpha(t)}^2)x-x}{t^2}+tO(t)=\frac{1}{12}D_{y}^2(x)
\]
which proves that $D_y^2(x)=[y,[y,x]]$ belongs to $H$ whenever $x$ and $y$ are in $H$.
\end{proof}

\section{Projections}

>From now on assume $M=e^H$ is a convex exponential set in $\Sigma$. As before, we identify the derivatives of all the geodesics at $p\in M$ with the tangent space of $M$ at $p$, in order to define the angles between curves and sets in a natural way: note that in this way, $T_1M=H$ and $T_pM=p^{\frac12}Hp^{\frac12}$ (which can be thought of as the parallel transport along the geodesic joining $1$ and $p$ in $M$).

In particular, $T_1\Sigma$ is naturally identified with ${\cal A}_h$ and the same is true for $T_p\Sigma$, for any $p\in \Sigma$, since $p^{\frac12}{\cal A}_hp^{\frac12}={\cal A}_h$ (this is clear also from the fact that $\Sigma$ is open in ${\cal A}_h$).

\smallskip

\begin{lemm} Let $r\in\Sigma$. There is at most one point $p=\Pi_M(r)$ in $M$ such that the geodesic joining $r$ and $p$ is orthogonal to $M$ at $p$
\end{lemm}
\begin{proof} Assume there are two points $p$ and $p'$ in $M$ and two vectors $v$ and $v'$ orthogonal to $M$ at $p$ and $p'$ respectively such that $\gamma_1(1)=Exp_p(v)=\gamma_2(1)=Exp_{p'}(v'):=r$, and consider the geodesic triangle with sides the given geodesics and the unique geodesic in $M$ joining $p$ and $p'$. Since the angles at $p$ and $p'$ are right angles, and the summ of the inner angles of any such geodesic triangle is less or equal than $\pi$, it must be that the angle at $r$ is zero: since geodesics are unique (given an initial velocity and an initial position $r$), it must be that $\gamma_1=\gamma_2$, hence $p=p'$ and $v=v'$. 
\end{proof}

\smallskip

Set $NM$ as the normal bundle of $M$, i.e. $NM=\{(p,v):p\in M, v\in (T_pM)^{\perp}\}$.

Consider the map $E: NM\to \Sigma$ given by $(p,v)\mapsto Exp_p(v)$; since $E$ is analytic and with the right identifications has differential (at $(p,0)$) the identity map, $E(NM)$ contains an open neighbourhood of $M$ in $\Sigma$ (with the norm topology).

\begin{lemm}
The map $\Pi_M: E(NM)\to M$ that assigns the endpoint of the minimizing geodesic is contractive for the geodesic metric.
\end{lemm}
\begin{proof}
If $r,s$ are two points in $E(NM)$, we will prove that this projection is contractive. Assume $\Pi_M(r)=p,\Pi_M(s)=q\in M$, $v\in T_p\Sigma$ is orthogonal to $M$ at $p$ and $w$ is orthogonal to $M$ at $q$; let's consider the distance function
\[
f(t)=dist^2(Exp_p(tv),Exp_q(tw))=dist^2(\gamma_1(t),\gamma_2(t))
\]
where $\gamma_1(t)$ is the only geodesic with initial velocity $v$ starting at $p$ and $\gamma_2$ is the only geodesic with initial speed $w$ starting at $q$. Namely,
\[
\gamma_1(t)=p^{\frac12} e^{t p^{-\frac12}vp^{-\frac12}}p^{\frac12}\qquad\mbox{ and }\qquad \gamma_2(t)=q^{\frac12} e^{t q^{-\frac12}wq^{-\frac12}}q^{\frac12}.
\]
Since $v\in \left(T_pM\right)^{\perp}$ and $w\in \left(T_qM\right)^{\perp}$, we have
\begin{equation}\label{ort}\begin{array}{l}
<v,p^{\frac12}xp^{\frac12}>_p=\tau\left( x p^{-\frac12}vp^{-\frac12}\right)=0\mbox{ for any }x\in H=T_1M \mbox{ and }\\
\\
<w,q^{\frac12}yq^{\frac12}>_q=\tau\left( y q^{-\frac12}wq^{-\frac12}\right)=0\mbox{ for any } y\in H=T_1M.
\end{array}
\end{equation}
Now we use the formula $dist(e^A,e^B)=\|\ln(e^{A/2}e^{-B}e^{A/2})\|_2$
for $A=\ln(\gamma_1(t))$ and $B=\ln(\gamma_2(t))$, to write
\[
f(t)=\|\ln(\gamma_1^{\frac12}\gamma_2^{-1}\gamma_1^{\frac12})\|^2_2=\tau\left(\ln^2( \gamma_1^{\frac12}\gamma_2^{-1}\gamma_1^{\frac12}  ) \right).
\]
Assume that $C$ is a simple, positively oriented curve in $\mathbb C$, around the spectrum of $\alpha_0=p^{\frac12}q^{-1}p^{\frac12}$. Then we can use the Cauchy formula to calculate $\ln^2(a)$ for any element $a\in \cal A$ such that $\sigma(a)\subset int(C)$, namely
\begin{equation}\label{cauchy}
\ln^2(a)=\frac{1}{2\pi i}\int_C \ln^2(z)(z-a)^{-1}dz.
\end{equation}
Naming $\alpha(t)=\gamma_1^{\frac12}(t)\gamma_2^{-1}(t)\gamma_1^{\frac12}(t)$, this formula holds true for $\alpha_0=\alpha(0)$ and for $\alpha(t)$ for $t$ sufficiently small. Note that
\[
f(t)=\tau\left(  \gamma^{-\frac12}(t)\gamma^{\frac12}(t) \ln^2(\alpha(t))   \right)=\tau\left( \gamma^{-\frac12}(t) \ln^2(\alpha(t) )\gamma^{\frac12}(t)   \right).
\]
If $x$ is invertible in $\cal A$, $xg(a)x^{-1}=g(xax^{-1})$ for any element $a\in \cal A$ and any analytic function $g$ in a neighbourhood of $\sigma(a)$. Then
\[
f(t)=\tau\left(  \ln^2\left[     \gamma_1(t)\gamma_2^{-1}(t)  \right]   \right)=\frac{1}{2\pi i}\int_C \ln^2(z)\, \tau\left[\left(z-\gamma_1(t)\gamma_2^{-1}(t) \right)^{-1}\right] \, dz.
\]
Now we compute $f'(0)$; note first that $\gamma_1(0)\gamma_2^{-1}(0)=pq^{-1}$ and also that
\[
\frac{d}{dt}_{t=0}\gamma_1(t)\gamma_2^{-1}(t)=-vq^{-1}+pq^{-1}wq^{-1}.
\]
Using the properties of the trace we get
\[
\frac{d}{dt}_{t=0}f(t)=-\frac{1}{2\pi i}\int_C \ln^2(z)\, \tau\left[       \left(z-pq^{-1}\right )^{-2} (-vq^{-1}+pq^{-1}wq^{-1})           \right] \, dz=
\]
\[
\qquad\qquad\quad=\tau\left[ \left(  -\frac{1}{2\pi i}\int_C \ln^2(z)\,     \left(z-pq^{-1}\right )^{-2} dz \right)  \left(-vq^{-1}+pq^{-1}wq^{-1}\right)           \right] .
\]
If we integrate by parts the first factor inside the trace, we obtain (note that $\frac{d}{dz}ln^2(z)=2\ln(z)z^{-1}=2z^{-1}\ln(z)$ and $C$ is a closed curve) that
\[\begin{array}{l}
\dot f(0)=\tau\left[ \left(    \frac{1}{2\pi i}\int_C 2\ln(z)z^{-1}\, \left(z-pq^{-1}\right )^{-1} dz  \right)  \left(-vq^{-1}+pq^{-1}wq^{-1}\right)       \right] = \\
\\
\qquad =-\tau\left[ \left(    \frac{1}{2\pi i}\int_C 2\ln(z)z^{-1}\, \left(z-pq^{-1}\right )^{-1} dz  \right)  vq^{-1}       \right] +\\
\\
\qquad\quad+ \tau\left[ \left(    \frac{1}{2\pi i}\int_C 2\ln(z)z^{-1}\, \left(z-pq^{-1}\right )^{-1} dz  \right)  pq^{-1}wq^{-1}    \right].
\end{array}
\]
Therefore,
\[\begin{array}{l}
\dot f(0)=-\frac{1}{2\pi i}\int\limits_C 2\ln(z)z^{-1}\, \tau\left[ q^{-1} \left(z-pq^{-1}\right )^{-1}  v  \right] dz +\\
\qquad\quad + \frac{1}{2\pi i}\int\limits_C 2\ln(z)z^{-1}\, \tau\left[ q^{-\frac12} \left(z-pq^{-1}\right )^{-1}  pq^{-1}wq^{-\frac12}  \right] dz.
\end{array}
\]
Using the elementary identities
\[ p^{\frac12}(z-p^{\frac12}q^{-1}p^{\frac12})^{-1}p^{-\frac12}=(z-pq^{-1})^{-1}=q^{\frac12}(z-q^{-\frac12}pq^{-\frac12})^{-1}q^{-\frac12}
\]
one arrives to the expression
\[\begin{array}{l}
\dot f(0)=-\frac{1}{2\pi i}\int_C 2\ln(z)z^{-1}\, \tau\left[ q^{-1} p^{\frac12}\left(z-p^{\frac12}q^{-1}p^{\frac12}\right )^{-1} p^{-\frac12} v  \right] dz +\\
\\
\quad\qquad+\frac{1}{2\pi i}\int_C 2\ln(z)z^{-1}\, \tau\left[ \left(z-q^{-\frac12}pq^{-\frac12}\right )^{-1} q^{-\frac12} pq^{-\frac12}q^{-\frac12}wq^{-\frac12}  \right] dz=\\
\\
\qquad= -2\tau\left[q^{-1}p^{\frac12}p^{-\frac12}qp^{-\frac12}\ln(p^{\frac12}q^{-1}p^{\frac12})p^{-\frac12}v\right]+\\
\\
\qquad\quad+2\tau\left[\ln(q^{-\frac12}pq^{-\frac12}) q^{\frac12}p^{-1}q^{\frac12}q^{-\frac12}pq^{-\frac12}q^{-\frac12}wq^{-\frac12}\right)=\\
\\
\qquad=-2\tau \left[ \ln(p^{\frac12}q^{-1}p^{\frac12}) p^{-\frac12}v p^{-\frac12}\right]+2\tau \left[\ln(q^{-\frac12}pq^{-\frac12}) q^{-\frac12}wq^{-\frac12}\right]=0+0=0,\\
\end{array}
\]
which holds by the orthogonality relations (\ref{ort}), naming $x=\ln(p^{\frac12}q^{-1}p^{\frac12})$ (recall that $M$ is convex), and $y=\ln(q^{-\frac12}pq^{-\frac12})$.

Since $f(t)$ is a convex function, $f$ has a global minimum at $t=0$, which proves that $dist(\Pi_M(r),\Pi_M(s))=dist(p,q)\le dist(r,s).$
\end{proof}

\begin{lemm}\label{normi}
Let $r\in\Sigma$ and let $M$ be a convex exponential set. The there exists $p\in M$ such that $dist(r,M)=dist(r,p)$  if and only if  there is a geodesic through $r$ orthogonal to $M$.
\end{lemm}
\begin{proof} Assume first that there is a point $p\in M$ such that the geodesic $\gamma$ through $p$ and $r$ is orthogonal to $M$ at $p$. Now for any point $q\in M$, take a geodesic $\beta$ joining $q$ and $r$, and a geodesic $\delta$ joining $p$ to $q$. Consider the geodesic triangle with sides $\gamma,\beta,\delta$; the angle opposite to $\beta$ is a right angle, so (see Lemma \ref{tri})
\[
Length(\beta)^2\ge Length(\delta)^2+Length(\gamma)^2\ge Length(\gamma)^2.
\]
This proves that $dist(p,r)\le dist(q,r)$ for any $q\in M$.

\smallskip

Assume now that $p\in M$ has the minimizing property and consider, for any point $q\in M$, the geodesic $\gamma_{p,q}(s)$ joining $p$ to $q$ (note that it is inside $M$ for any $s$ by virtue of the convexity). Now consider the family of geodesics
\[
\gamma_s(t)=\gamma_{r,\gamma_{pq}(s)}(t)=r^{\frac12}\left(   r^{-\frac12} p^{\frac12} \left( p^{-\frac12} qp^{-\frac12}  \right)^s p^{\frac12} r^{-\frac12}   \right)^t r^{\frac12}
\]
that is, the family of geodesics joining $r$ to $\gamma_{p,q}(s)$.

Put $g(s)=Length(\gamma_s)^2=dist(r,\gamma_{p,q}(s))$. This function has a minimum at $s=0$, hence (since it is $C^{\infty}$) it must be that $\dot g(0)=0$. As in the proof of the previous theorem, we have
\[
g(s)= \tau\left( \ln^2(r\gamma_{p,q}(s)^{-1}) \right)=\frac{1}{2\pi i}\int_C \ln^2(z)\, \tau\left[   \left(z-r\gamma_{p,q}^{-1}(s)\right )^{-1}  \right] \, dz.
\]
Taking the derivative at $s=0$ and integrating by parts we obtain
\[
0=\dot g(0)=-2\tau\left(\ln(rp^{-1})\ln(qp^{-1})p^{-1}\right)=2\tau\left(\ln(qp^{-1})p^{-1}\ln(p r^{-1})\right).
\]
On the other hand, the angle subtended by $\gamma_{p,q}$ and $\gamma_{r,p}$ at $p$ is
\[
<\dot\gamma_{r,p}(1),\dot\gamma_{p,q}(0)>_p=\tau\left(\ln(qp^{-1})p^{-1}\ln(pr^{-1})\right).
\]
This proves that $\gamma_{r,p}$ is orthogonal to any geodesic at $p$ contained in $M$, and by definition, it is orthogonal to $M$.
\end{proof}

\smallskip

The following is related to the main result in \cite{coco} by H. Porta and L. Recht:

\begin{theo}\label{proje}
If $M=e^H$ is a convex exponential set, and there is a closed, orthogonal supplement $S$ for $H$ (namely ${\cal A}_h=H\oplus_{\perp_{\tau}} S$) then for any point $r\in\Sigma$ there is a geodesic through $r$ orthogonal to $M$.
\end{theo}
\begin{proof}
Exactly as in \cite{coco}, there is an equality of sets $E(NM)=\Sigma$, where $NM$ stands for the normal bundle of $M$, i.e the pairs $(p,v)$ with $p\in M$ and $v\perp_p M$.
\end{proof}

The typical examples for this situation arise when $H=B_h$ for a subalgebra $B$ of ${\cal A}$. In this case, by a result of Takesaki \cite{take}, there is a conditional expectation ${\cal E}:{\cal A}\to{\cal A}$  with rank $B$, compatible with $\tau$ (i.e $\tau({\cal E}(x))=\tau(x)$ for any $x\in{\cal A}$).

\smallskip

\begin{coro}
If $H$ is a closed subspace in ${\cal A}_h$ (supplemented as in the previous theorem) such that $[x,[x,y]]\in H$ whenever $x,y\in H$, then for any $z\in{\cal A}_h$ we can factor
\[
e^z=e^ye^we^y
\]
for unique $y\in H$ and $w\in {\cal A}_h$ such that $\tau(wx)=0$ for any $x\in H$. Moreover, $e^{2y}$  minimizes (geodesic) distance between $M=e^H$ and $e^z$, and is unique with that property.
\end{coro}

\begin{coro}
Fix $\cal D$ a m.a.s.a of ${\cal A}$. Then for any $x\in{\cal A}_h$ there are unique $d\in{\cal D}_h$ and $v\in {\cal A}_h$ such that $\tau(vz)=0$ for any $z\in \cal D$, and $e^x=de^vd$
\end{coro}

\begin{coro}\label{invert}
If $H$ is a closed, supplemented subspace in ${\cal A}_h$ such that $[x,[x,y]]\in H$ whenever $x,y\in H$, then for any $g\in G_{\cal A}$ we can factor
\[
g=e^xe^yu
\]
for unique $x\in H, y\in H^{\perp}$ and $u$ in the unitary group of ${\cal A}$.
\end{coro}
\begin{proof}
Note that $gg^*\in \Sigma$ hence $gg^*=e^xe^{2y}e^x$ where $x,y$ are as required. Now take $u=e^{-y}e^{-x}g$; a straightforward computation shows that $uu^*=u^*u=1$. Uniqueness follows from the uniqueness of $x,y$.
\end{proof}

\bigskip

\noindent
Esteban Andruchow and Gabriel Larotonda\\
Instituto de Ciencias \\
Universidad Nacional de Gral. Sarmiento \\
J. M. Gutierrez 1150 \\
(1613) Los Polvorines \\
Argentina  \\
e-mail: eandruch@ungs.edu.ar, glaroton@ungs.edu.ar

\end{document}